%% file: 2010-SO-equiv_conc_ineq.tex
\documentclass[reqno]{amsart}

\usepackage{draftwatermark}
\SetWatermarkLightness{0.9}
\SetWatermarkText{---~PREPRINT~---}
\SetWatermarkScale{4}

\usepackage{amsmath} %% Standard AMS stuff
\usepackage{amssymb} %% Standard AMS stuff
\usepackage{amsthm} %% Standard AMS stuff
\usepackage{mathrsfs} %% Ralph Smith Formal Script
\usepackage{bbm} %% Improved Blackboard Bold: \mathbbm
\usepackage{empheq}
\usepackage{graphicx}
\usepackage[usenames,dvipsnames]{pstricks}
\usepackage{pst-plot}
\usepackage{paralist}

\usepackage[colorlinks]{hyperref} %% Hyperlinks
\hypersetup{
  pdftitle = {Equivalence of concentration inequalities for linear and non-linear functions},
  pdfauthor = {T. J. Sullivan \& H. Owhadi},
  pdfkeywords = {concentration of measure, large deviations, quasiconvexity, normal distance},
  pdfsubject = {2010 MSC: 60E15, 60F10, 52A07},
}

\newcommand{\Ball}{\mathbb{B}}
\newcommand{\COMMENT}[1]{}
\newcommand{\CH}{\mathop{\mathrm{co}}}
\newcommand{\CCH}{\mathop{\overline{\mathrm{co}}}}
\renewcommand{\d}{\mathrm{d}}

\newcommand{\eps}{\varepsilon}
\newcommand{\Exp}{\mathbb{E}}
\renewcommand{\H}{\mathbb{H}}
\newcommand{\Med}{\mathbb{M}}
\newcommand{\N}{\mathbb{N}}
\newcommand{\one}{\mathbbm{1}}
\renewcommand{\P}{\mathbb{P}}
\newcommand{\R}{\mathbb{R}}
\newcommand{\smid}{\,\middle|\,}

\newtheorem{thm}{Theorem}[section]
\newtheorem{prop}[thm]{Proposition}
\newtheorem{lem}[thm]{Lemma}

\theoremstyle{definition}
\newtheorem{defn}[thm]{Definition}
\newtheorem{rmk}[thm]{Remark}
\newtheorem{eg}[thm]{Example}

\title[Concentration inequalities for linear and non-linear functions]{
	Equivalence of concentration inequalities for linear and non-linear functions
}

\author{T.\ J.\ Sullivan}
\address{
	Graduate Aerospace Laboratories \\
	California Institute of Technology \\
	Mail Code 205-45 \\
	1200 East California Boulevard \\
	Pasadena \\
	CA 91125 \\
	United States of America
}
\email{tjs@caltech.edu}
\urladdr{\url{http://www.its.caltech.edu/~tjs/}}

\author{H.\ Owhadi}
\address{
	Applied \& Computational Mathematics and Control \& Dynamical Systems \\
	California Institute of Technology \\
	Mail Code 217-50 \\
	1200 East California Boulevard \\
	Pasadena \\
	CA 91125 \\
	United States of America
}
\email{owhadi@caltech.edu}
\urladdr{\url{http://www.acm.caltech.edu/~owhadi/}}

\date{\today}



\thanks{The authors acknowledge portions of this work supported by the United States Department of Energy National Nuclear Security Administration under Award Number DE-FC52-08NA28613 through the California Institute of Technology's ASC/PSAAP Center for the Predictive Modeling and Simulation of High Energy Density Dynamic Response of Materials.}

\makeatletter
\@addtoreset{equation}{section}
\@addtoreset{figure}{section}
\makeatother

\begin{document}

\begin{abstract}
	We consider a random variable $X$ that takes values in a (possibly infinite-dimensional) topological vector space $\mathcal{X}$.  We show that, with respect to an appropriate ``normal distance'' on $\mathcal{X}$, concentration inequalities for linear and non-linear functions of $X$ are equivalent.  This normal distance corresponds naturally to the concentration rate in classical concentration results such as Gaussian concentration and concentration on the Euclidean and Hamming cubes.  Under suitable assumptions on the roundness of the sets of interest, the concentration inequalities so obtained are asymptotically optimal in the high-dimensional limit.
\end{abstract}

\maketitle

\section{Introduction}
\label{sec:intro}

It is by now almost classical that smooth enough convex functions enjoy good concentration properties;  see \emph{e.g.}\ \cite{Ledoux:2001} \cite{Lugosi:2009} \cite{McDiarmid:1998} \cite{MilmanSchechtman:1986} for surveys of the literature.  It is also known that convexity can be neglected in the Gaussian case and that the smoothness assumptions are not essential and can be replaced, for instance, with bounded martingale differences;  see \emph{e.g.}\ \cite{McDiarmid:1989} \cite{McDiarmid:1997} and also \cite{Vu:2002}.

A common feature of many concentration results is that an appropriate notion of distance is needed, \emph{e.g.}\ Talagrand's convex distance \cite{Talagrand:1995}.  In this paper, a notion of ``normal distance'' on a topological vector space $\mathcal{X}$ is introduced through a technique commonly used in large deviations theory, Chernoff bounding, \emph{i.e.}\ estimating the measure of a set by using a containing half-space.  Although simple, this method leads to a notion of distance that is in some sense ``natural'' with respect to the duality structure on $\mathcal{X}$.  Remarkably, with respect to this distance, concentration inequalities on the tails of linear, convex, quasiconvex and non-linear functions on $\mathcal{X}$ are mutually equivalent.

Concentration of measure is based on a simple but non-trivial observation originally due to L{\'e}vy \cite{Levy:1951}:  in a high-dimensional probability space, ``nearly all'' the probability mass lies close to any set with measure at least $\frac{1}{2}$;  put another way, functions of many independent variables with small sensitivity to each individual input are very nearly constant.  A typical concentration inequality is of the form
\begin{equation}
	\label{eq:deviation_ineq}
	\P[ | f(X) - m | \geq r ] \leq C_{1} \exp ( - C_{2} r^{2} ),
\end{equation}
where $f$ is a suitably well-behaved function, $X$ is a random variable such that the push-forward measure $(f \circ X)_{\ast} \P$ has some concentration property, and $m$ is either the mean value $\Exp[f(X)]$ or median value $\Med[f(X)]$;  sometimes the control is one-sided, and the absolute value in \eqref{eq:deviation_ineq} is omitted.  A notable feature of this paper is that it provides concentration inequalities with $m = f(\Exp[X])$.

The key property of the normal distance of this paper is contained in the following \emph{portmanteau theorem} for the equivalence of various concentration inequalities with respect to normal distance:

\begin{thm}
	\label{thm:portmanteau}
	Let $\mathcal{X}$ be a real topological vector space and $\mathcal{X}^{\ast}$ its continuous dual space.  Let $\Psi \colon \mathcal{X}^{\ast} \to [0, +\infty]$ be positively homogeneous of degree one.  Define the \emph{$\Psi$-normal distance} from $x \in \mathcal{X}$ to $A \subseteq \mathcal{X}$ by
	\[
		d_{\perp, \Psi}(x, A) := \sup \left\{ \frac{\langle \nu, x - p \rangle_{+}}{\Psi(\nu)} \smid \begin{matrix} p \in \mathcal{X} \text{ and } \nu \in \mathcal{X}^{\ast} \text{such that,} \\ \text{for all $a \in A$, } \langle \nu, a \rangle \leq \langle \nu, p \rangle \end{matrix} \right\},
	\]
	with the convention that $0 / 0 = 0$.  Then the following statements about any random variable $X$ that takes values in $\mathcal{X}$ are equivalent:
	\begin{compactenum}[(i)]
		\item \label{port-1} for every closed half-space $\H_{p, \nu} := \{ x \in \mathcal{X} \mid \langle \nu, x - p \rangle \leq 0 \} \subseteq \mathcal{X}$, where $p \in \mathcal{X}$ and $\nu \in \mathcal{X}^{\ast}$,
		\[
			\P[X \in \H_{p, \nu}] \leq \exp \left( - \frac{d_{\perp, \Psi}(\Exp[X], \H_{p, \nu})^{2}}{2} \right);
		\]		
		\item \label{port-2} for every convex set $K \subseteq \mathcal{X}$,
		\[
			\P[X \in K] \leq \exp \left( - \frac{d_{\perp, \Psi}(\Exp[X], K)^{2}}{2} \right);
		\]
		\item \label{port-3} for every measurable $A \subseteq \mathcal{X}$,
		\[
			\P[X \in A] \leq \exp \left( - \frac{d_{\perp, \Psi}(\Exp[X], A)^{2}}{2} \right);
		\]
		\item \label{port-4} for every measurable $f \colon \mathcal{X} \to \R \cup \{ \pm \infty \}$ and every $\theta \in \R \cup \{ \pm \infty \}$,
		\[
			\P[f(X) \leq \theta] \leq \exp \left( - \frac{d_{\perp, \Psi} (\Exp[X], f^{-1}([-\infty, \theta]))^{2}}{2} \right);
		\]
		\item \label{port-5} for every quasiconvex $f \colon \mathcal{X} \to \R \cup \{ \pm \infty \}$ and every $\theta \in \R \cup \{ \pm \infty \}$,
		\[
			\P[f(X) \leq \theta] \leq \exp \left( - \frac{d_{\perp, \Psi}(\Exp[X], f^{-1}([-\infty, \theta]))^{2}}{2} \right).
		\]
	\end{compactenum}
\end{thm}

Note that if $f$ is quasilinear (\emph{i.e.}\ both $f$ and $-f$ are quasiconvex), then formulation (\ref{port-5}) yields concentration inequalities for both the lower and upper tails of $f(X)$.

The notation and setting of the paper are covered in section \ref{sec:notation_and_background}, along with a review of some definitions and results from the concentration-of-measure literature.  Normal distance is defined and its properties (including theorem \ref{thm:portmanteau}) are examined in section \ref{sec:normal_distance}.  In section \ref{sec:main}, the normalizing function $\Psi$ is determined explicitly in several cases, thereby connecting theorem \ref{thm:portmanteau} with classical concentration results. In particular, proposition \ref{prop:emp_mean} identifies the normal distance that corresponds to the concentration of a vector, the entries of which are the empirical (sampled) means of functions of independent random variables.  In section \ref{sec:asymptotics}, it is shown that the inequality in theorem \ref{thm:portmanteau}(\ref{port-3}) is asymptotically sharp (in the sense used in large deviations theory) in the high-dimensional limit, provided that $A$ is convex and ``sufficiently round'' at those points of $A$ that are closest to the center of mass $\Exp[X]$.  Finally, for completeness, the method of Chernoff bounds and its consequences for convex sets are reviewed in an appendix (section \ref{sec:Chernoff}).

\section{Notation and Background}
\label{sec:notation_and_background}

Let $\mathcal{X}$ be a real topological vector space.  Let $\mathcal{X}^{\ast}$ denote the continuous dual space of $\mathcal{X}$ and let $\langle \ell, x \rangle$ denote the dual pairing between $\ell \in \mathcal{X}^{\ast}$ and $x \in \mathcal{X}$;  $\langle v, \ell \rangle$ will also denote the dual pairing between $v \in \mathcal{X}^{\ast \ast}$ and $\ell \in \mathcal{X}^{\ast}$.  It is not strictly necessary to assume that $\mathcal{X}$ is locally convex, but the results of this paper may be trivially true if $\mathcal{X}^{\ast}$ does not contain enough linear functionals.

\subsection{Half-Spaces}  

Given $p \in \mathcal{X}$ and $\nu \in \mathcal{X}^{\ast}$, $\H_{p, \nu}$ will denote the closed half-space of $\mathcal{X}$ that has $p$ in its frontier and outward-pointing normal $\nu$, \emph{i.e.}
\begin{equation}
	\label{eq:halfspace}
	\H_{p, \nu} := \left\{ x \in \mathcal{X} \smid \langle \nu, x \rangle \leq \langle \nu, p \rangle \right\}.
\end{equation}
%See figure \ref{fig:halfspace} for an illustration.
Note well the degenerate case $\H_{p, 0} = \mathcal{X}$.  Every $(p, \nu) \in \mathcal{X} \times \mathcal{X}^{\ast}$ defines a unique closed half-space of $\mathcal{X}$, whereas a given closed half-space can have multiple distinct representations:  $\H_{p, \nu} = \H_{p', \nu'}$ if, and only if, $\nu$ is a positive multiple of $\nu'$ and $\langle \nu, p - p' \rangle = \langle \nu', p - p' \rangle = 0$.

\subsection{Convex Sets and Cones}  

The closed convex hull of $A \subseteq \mathcal{X}$ will be denoted by $\CCH(A)$.  Given a closed convex set $K \subseteq \mathcal{X}$ and $p \in K$, $\mathrm{N}_{p}^{\ast} K$ denotes the \emph{outward normal cone} to $K$ at $p$, and $\mathrm{N}^{\ast} K$ denotes the \emph{outward normal bundle} of $K$:
\begin{equation}
	\label{eq:normal_cone}
	\mathrm{N}_{p}^{\ast} K := \left\{ \nu \in \mathcal{X}^{\ast} \smid K \subseteq \H_{p, \nu} \right\},
\end{equation}
\begin{equation}
	\label{eq:normal_bundle}
	\mathrm{N}^{\ast} K := \left\{ (p, \nu) \in \mathcal{X} \times \mathcal{X}^{\ast} \smid p \in K, \nu \in \mathrm{N}_{p}^{\ast} K \right\}.
\end{equation}
The outward normal cone $\mathrm{N}_{p}^{\ast} K$ is a \emph{pointed convex cone}:  it contains $0$, is convex, and $s_{1} \nu_{1} + s_{2} \nu_{2} \in \mathrm{N}_{p}^{\ast} K$ for all $s_{1}, s_{2} \geq 0$ and all $\nu_{1}, \nu_{2} \in \mathrm{N}_{p}^{\ast} K$.  Also, $\mathrm{N}_{p}^{\ast} K = \{ 0 \}$ if $p$ is an interior point of $K$.  Note that $\mathrm{N}^{\ast} K \subseteq \mathcal{X} \times \mathcal{X}^{\ast}$ is not necessarily a convex set.  See figure \ref{fig:normal_cone} for an illustration.

\begin{figure}
	\noindent\framebox[\linewidth]{\parbox{\linewidth-20pt}{
		\begin{center}
			\scalebox{1} % Change this value to rescale the drawing.
			{
				\input ./fig-normal_cones.tex
			}
		\end{center}
		\caption{A convex set $K$ and its outward normal cones at points $p, q, r \in K$.  $\partial K$ is smooth at $p \in \partial K$, so $\mathrm{N}_{p}^{\ast} K$ is a half-line;  $\partial K$ has a vertex at $q$, so $\mathrm{N}_{q}^{\ast} K$ is a pointed convex cone with non-empty interior;  at the interior point $r$, $\mathrm{N}_{r}^{\ast} K$ is the empty set.}
		\label{fig:normal_cone}
	}}
\end{figure}

\subsection{Quasiconvexity}  

If $K \subseteq \mathcal{X}$ is a convex set, then a function $f \colon K \to \R \cup \{ \pm \infty \}$ is said to be \emph{quasiconvex} if, for every $\theta \in \R \cup \{ \pm \infty \}$, the sublevel set
\begin{equation}
	\label{eq:sublevel}
	f^{-1}([-\infty, \theta]) := \{ x \in K \mid - \infty \leq f(x) \leq \theta \}
\end{equation}
is a convex set;  equivalently, $f$ is quasiconvex if, for all $x, y \in K$ and $t \in [0, 1]$,
\begin{equation}
	\label{eq:equivalent_quasiconvex}
	f((1 - t)x + t y) \leq \max \{ f(x), f(y) \}.
\end{equation}
$f$ is said to be \emph{quasiconcave} if $-f$ is quasiconvex, and $f$ is said to be \emph{quasilinear} if it is both quasiconvex and quasiconcave.  Every convex (resp.\ concave, linear) function is quasiconvex (resp.\ quasiconcave, quasilinear), but not \emph{vice versa}.  In particular, a function $f \colon \R^{N} \to \R$ is quasilinear if, and only if, it is the composition of a monotone function with a linear functional on $\R^{N}$ \cite[p.\ 122]{BoydVandenberghe:2004}.

\subsection{Indicator and Characteristic Functions.}  

Given a set $A \subseteq \mathcal{X}$, $\one_{A}$ and $\chi_{A}$ denote its \emph{indicator function} and \emph{characteristic function} respectively:
\begin{equation}
	\label{eq:indicator_fn}
	\one_{A}(x) :=
	\begin{cases}
		1, & \text{if $x \in A$,} \\
		0, & \text{if $x \notin A$;}
	\end{cases}
\end{equation}
\begin{equation}
	\label{eq:characteristic_fn}
	\chi_{A}(x) :=
	\begin{cases}
		0, & \text{if $x \in A$,} \\
		+\infty, & \text{if $x \notin A$.}
	\end{cases}
\end{equation}
Note that, for any convex set $K \subseteq \mathcal{X}$, $\chi_{K}$ is a convex function.

\subsection{Probabilistic Notions}  

Let $(\Omega, \mathscr{F}, \P)$ be a probability space and let $X \colon \Omega \to \mathcal{X}$ be an $\mathcal{X}$-valued random variable.  $\Exp[\cdot]$ denotes the expectation operator with respect to the probability measure $\P$:  $\Exp[X]$ is defined to be any $m \in \mathcal{X}$ such that
\begin{equation}
	\label{eq:expectation_vector}
	\Exp[\langle \ell, X - m \rangle] \equiv \int_{\Omega} \langle \ell, X(\omega) - m \rangle \, \d \P (\omega) = 0 \text{ for all } \ell \in \mathcal{X}^{\ast};
\end{equation}
if $\mathcal{X}^{\ast}$ separates the points of $\mathcal{X}$ (\emph{e.g.}\ if $\mathcal{X}$ is a Banach space), then $\Exp[X]$ is unique.  For $Y \colon \Omega \to \R$, any $m \in \R$ that satisfies
\begin{equation}
	\label{eq:median}
	\sup \left\{ v \in \R \smid \P[Y \leq v] \leq \frac{1}{2} \right\} \leq m \leq \inf \left\{ v \in \R \smid \P[Y \leq v] \geq \frac{1}{2} \right\}
\end{equation}
will be called a \emph{median} of $Y$ and denoted $\Med[Y]$.  $M_{X} \colon \mathcal{X}^{\ast} \to [0, + \infty]$ denotes the \emph{moment-generating function} defined by
\begin{equation}
	\label{eq:MGF}
	M_{X}(\ell) := \Exp \left[ \exp \langle \ell, X \rangle \right] \text{ for all } \ell \in \mathcal{X}^{\ast}.
\end{equation}
$\Lambda_{X} \colon \mathcal{X}^{\ast} \to \R \cup \{ \pm \infty \}$ denotes the \emph{cumulant generating function} (or \emph{logarithmic moment-generating function}) defined by
\begin{equation}
	\label{eq:logMGF}
	\Lambda_{X}(\ell) := \log M_{X}(\ell) = \log \Exp \left[ \exp \langle \ell, X \rangle \right] \text{ for all } \ell \in \mathcal{X}^{\ast}.
\end{equation}
By H{\"o}lder's inequality, $\Lambda_{X}$ is a convex function.

\subsection{Talagrand's Inequalities}

It has been known for some time that convex sets and functions enjoy good concentration properties;  moreover, to get good concentration results, it is necessary to measure distances in the right way.

For example, a theorem of Talagrand shows that if a convex set $K \subseteq \mathbb{R}^{N}$ occupies a ``significant'' portion of the Hamming cube $\{ -1, +1 \}^{N}$ and $t \gg 1$, then ``nearly all'' of the points of the Hamming cube lie within Euclidean distance $t$ of $K$.  Define the \emph{Euclidean Hausdorff distance} from $x \in \R^{N}$ to $A \subseteq \R^{N}$ by
\begin{equation}
	\label{eq:Hausdorff_distance}
	d_{\mathrm{Haus}}(x, A) := \inf \{ \| x - a \|_{2} \mid a \in A \}.
\end{equation}
Talagrand \cite{Talagrand:1988} showed that if $X$ is uniformly distributed in $\{ -1, +1 \}^{N}$ then, for any $A \subseteq \R^{N}$, $\Exp[\exp(d_{\mathrm{Haus}}(X, \CCH(A))^{2} / 8)] \leq \P[X \in A]^{-1}$;  hence, Chebyshev's inequality implies that, for any $t \geq 0$,
\begin{equation}
	\label{eq:Talagrand-1}
	\P[X \in A] \P [d_{\mathrm{Haus}}(X, \CCH(A)) \geq t] \leq \exp \left( - \frac{t^{2}}{8} \right).
\end{equation}

More interesting results can be obtained if one uses not the Euclidean distance but the Hamming distance --- or, more accurately, an infimum over weighted Hamming distances.  For $w = (w_{1}, \dots, w_{N}) \in [0, + \infty)^{N}$, define the \emph{$w$-weighted Hamming distance} $d_{w}$ on a product of sets $\mathcal{X} = \prod_{n = 1}^{N} \mathcal{X}_{n}$ by
\begin{equation}
	\label{eq:w_Hamming_distance}
	d_{w}(x, y) := \sum_{n = 1}^{N} w_{n} (1 - \delta_{x_{n}, y_{n}});
\end{equation}
that is, $d_{w}(x, y)$ is the $w$-weighted sum of the number of components in which $x, y \in \mathcal{X}$ differ.  For $x \in \mathcal{X}$ and $A \subseteq \mathcal{X}$, set $d_{w}(x, A) := \inf_{a \in A} d_{w}(x, a)$.  Define \emph{Talagrand's convex distance} from $x \in \mathcal{X}$ to $A \subseteq \mathcal{X}$ by
\begin{equation}
	\label{eq:Talagrand_distance}
	d_{\mathrm{Tal}}(x, A) := \sup \left\{ d_{w}(x, A) \smid w \in [0, + \infty)^{N}, \sum_{n = 1}^{N} w_{n}^{2} = 1 \right\},
\end{equation}
and, for $A, B \subseteq \mathcal{X}$, let $d_{\mathrm{Tal}}(A, B) := \inf_{a \in A} d_{\mathrm{Tal}}(a, B)$.  Talagrand \cite[\S4.1]{Talagrand:1995} showed that if $X = (X_{1}, \dots, X_{N})$ is any $\mathcal{X}$-valued random variable with independent components, then
\begin{equation}
	\label{eq:Talagrand-3}
	\P[X \in A] \P[X \in B] \leq \exp \left( - \frac{d_{\mathrm{Tal}}(A, B)^{2}}{4} \right).
\end{equation}

These bounds on the probabilities of sets lead to deviation inequalities for convex Lipschitz functions.  For example (\emph{cf.}\ \cite{JohnsonSchechtman:1991} \cite{Talagrand:1988}), let $X$ be any random variable in the unit cube in $\R^{N}$ with independent components, and let $f \colon [0, 1]^{N} \to \R$ be convex and Lipschitz with $\| f \|_{\mathrm{Lip}} \leq 1$;  then, for any $t \geq 0$,
\begin{equation}
	\label{eq:Talagrand-2}
	\P[f(X) \geq \Med[f(X)] + t] \leq 2 \exp \left( - \frac{t^{2}}{4} \right).
\end{equation}	

Note, however, that these results use not only the convexity of the function of interest, but also require Lipschitz continuity.  What concentration inequalities can be shown to hold without smoothness assumptions?

\subsection{McDiarmid's Inequality}

One smoothness-free concentration inequality is \emph{McDiarmid's inequality} \cite{McDiarmid:1989}, also known as the \emph{bounded differences inequality}, which itself generalizes an earlier inequality of Hoeffding \cite{Hoeffding:1963}.  McDiarmid's inequality is by no means the strongest concentration-of-measure inequality in the literature, but is useful because of its simple hypotheses and proof.

Define the \emph{McDiarmid diameter} of $f$, denoted $\mathcal{D}[f]$, by
\begin{equation}
	\label{eq:McD_diameter}
	\mathcal{D}[f] := \left( \sum_{n = 1}^{N} \mathcal{D}_{n}[f]^{2} \right)^{1/2},
\end{equation}
where the $n^{\mathrm{th}}$ \emph{McDiarmid subdiameter} $\mathcal{D}_{n}[f]$ is defined by
\begin{equation}
	\label{eq:McD_subdiameter}
	\mathcal{D}_{n}[f] := \sup \{ | f(x) - f(y) | \mid x_{j} = y_{j} \text{ for } j \neq n \}.
\end{equation}
When $\Exp[|f(X)|]$ is finite and $X_{1}, \dots, X_{N}$ are independent, McDiarmid's inequality bounds the deviations of $f(X)$ from $\Exp[f(X)]$ in terms of the McDiarmid diameter of $f$:  for any $r > 0$,
\begin{subequations}
	\label{eq:McD}
	\begin{align}
		\label{eq:McD-leq}
		\P[f(X) - \Exp[f(X)] \leq - r] & \leq \exp \left( - \frac{2 r^{2}}{\mathcal{D}[f]^{2}} \right), \\
		\label{eq:McD-geq}
		\P[f(X) - \Exp[f(X)] \geq r] & \leq \exp \left( - \frac{2 r^{2}}{\mathcal{D}[f]^{2}} \right).
	\end{align}
\end{subequations}
McDiarmid's inequality implies that, for any $\theta \in \R \cup \{ \pm \infty \}$,
\begin{subequations}
	\label{eq:McD2}
	\begin{align}
		\label{eq:McD2-leq}
		\P[f(X) \leq \theta] & \leq \exp \left( - \frac{2 (\Exp[f(X)] - \theta)_{+}^{2}}{\mathcal{D}[f]^{2}} \right), \\
		\label{eq:McD2-geq}
		\P[f(X) \geq \theta] & \leq \exp \left( - \frac{2 (\theta - \Exp[f(X)])_{+}^{2}}{\mathcal{D}[f]^{2}} \right).
	\end{align}
\end{subequations}
McDiarmid's inequality (and similar inequalities such as martingale inequalities) have the advantage that a bound on the tails of $f(X)$ is obtained solely in terms of the mean output $\Exp[f(X)]$ and the McDiarmid diameter $\mathcal{D}[f]$.  However, McDiarmid's inequality cannot take advantage of any other properties of $f$ such as convexity or monotonicity;  furthermore, if $f$ has infinite McDiarmid diameter on the essential range of $X$, then the trivial upper bound $1$ is obtained.

There are many other sources of concentration-of-measure inequalities:  these include logarithmic Sobolev inequalities and the Herbst argument \cite{BakryEmery:1985} \cite{Gross:1975} \cite{HolleyStroock:1987}, the entropy method \cite{BobkovLedoux:1998} \cite{BoucheronLugosiMassart:2003} \cite{Ledoux:1996}, and information-theoretic methods \cite{Dembo:1997} \cite{Marton:1996}.  Of particular interest are those concentration results that apply to infinite-dimensional settings \cite{LedouxTalagrand:1991}.

\section{Normal Distance}
\label{sec:normal_distance}

As noted above, efficient presentation of many concentration-of-measure inequalities relies on having an appropriate notion of function variation (\emph{e.g.}\ the Lipschitz norm or McDiarmid diameter) or distance (\emph{e.g.}\ Talagrand's convex distance).  The inequalities that will be established in section \ref{sec:main} can be phrased in terms of transforms of moment-generating functions, but are more transparent if phrased in terms of a \emph{normal distance}, which will introduced in this section.

Fix a function $\Psi \colon \mathcal{X}^{\ast} \to [0, + \infty]$ that is positively homogeneous of degree one, \emph{i.e.}\ such that $\Psi(\alpha \ell) = \alpha \Psi(\ell)$ for all $\alpha \geq 0$ and all $\ell \in \mathcal{X}^{\ast}$.  By analogy with the situation in finite-dimensional Euclidean space, in which $\Psi = \| \cdot \|_{2}$ on $(\R^{N})^{\ast}$, define the distance from a point $x \in \mathcal{X}$ to a half-space $\H_{p, \nu} \subseteq \mathcal{X}$ by
\begin{equation}
	\label{eq:normal_distance_halfspace}
	d_{\perp, \Psi}(x, \H_{p, \nu}) := \frac{\langle \nu, x - p \rangle_{+}}{\Psi(\nu)},
\end{equation}
with the convention that $0 / 0 = 0$, since the distance from $x \in \mathcal{X}$ to the trivial half-space $\H_{p, \nu} = \mathcal{X}$ ought to be zero.  Note that $d_{\perp, \Psi}(x, \H_{p, \nu}) = 0$ whenever $x \in \H_{p, \nu}$;  note also that the homogeneity assumption on $\Psi$ ensures that \eqref{eq:normal_distance_halfspace} is an unambiguous definition.   We now generalize \eqref{eq:normal_distance_halfspace} to more general subsets of $\mathcal{X}$ than half-spaces.  The heuristic is that the distance from $x$ to $A \subseteq \mathcal{X}$ should be the greatest possible distance (in the sense of \eqref{eq:normal_distance_halfspace}) from $x$ to any half-space that contains $A$;  the existence of the degenerate half-space $\H_{p, 0}$ ensures that the normal distance is zero if there are no proper half-spaces that contain $A$.

\begin{defn}
	\label{defn:normal_distance}
	Let $x \in \mathcal{X}$ and $A \subseteq \mathcal{X}$.  The \emph{$\Psi$-normal distance} from $x$ to $A$, denoted $d_{\perp, \Psi}(x, A)$, is defined (with the same convention that $0 / 0 = 0$) by
	\begin{equation}
		\label{eq:normal_distance}
		d_{\perp, \Psi}(x, A) := \sup \left\{ \frac{\langle \nu, x - p \rangle_{+}}{\Psi(\nu)} \smid \begin{matrix} p \in \mathcal{X} \text{ and } \nu \in \mathcal{X}^{\ast} \\ \text{such that } A \subseteq \H_{p, \nu} \end{matrix} \right\}.
	\end{equation}
	The $\Psi$-normal distance from $A \subseteq \mathcal{X}$ to $B \subseteq \mathcal{X}$ is defined by $d_{\perp, \Psi}(A, B) := \inf_{a \in A} d_{\perp, \Psi}(a, B)$.  In the special case $\mathcal{X} = \R^{N}$ and $\Psi = \| \cdot \|_{2}$ on $(\R^{N})^{\ast}$, we shall simply write $d_{\perp}$ for $d_{\perp, \Psi}$, \emph{i.e.}
	\begin{equation}
		\label{eq:normal_distance_Euclidean}
		d_{\perp}(x, A) := \sup \left\{ \frac{(\nu \cdot (x - p))_{+}}{\| \nu \|_{2}} \smid \begin{matrix} p \in \R^{N} \text{ and } \nu \in (\R^{N})^{\ast} \\ \text{such that } A \subseteq \H_{p, \nu} \end{matrix} \right\}.		
	\end{equation}
\end{defn}

Note well that the definition of the normal distance $d_{\perp, \Psi}(x, A)$ does not require $\mathcal{X}$ to be normed;  even when $\mathcal{X}$ is equipped with a norm $\| \cdot \|_{\mathcal{X}}$ and $\Psi$ is the corresponding operator norm, the normal distance $d_{\perp, \Psi}(x, A)$ is not the same as the Hausdorff distance from $x$ to $A$ defined by
\begin{equation}
	\label{eq:Hausdorff_distance_X}
	d_{\mathrm{Haus}}(x, A) := \inf \{ \| x - a \|_{\mathcal{X}} \mid a \in A \};
\end{equation}
see figure \ref{fig:normal_distance} for an illustration.  Note also that it is not generally true that $d_{\perp, \Psi}(A, B) = d_{\perp, \Psi}(B, A)$:  consider \emph{e.g.}\ $B := \{ (0, 1) \}$ and $A$ as in figure \ref{fig:normal_distance}, in which case
\[
	d_{\perp, \Psi}(A, B) = \inf_{a \in A} d_{\perp, \Psi}(a, B) = 1 \neq 0 = d_{\perp, \Psi}(B, A).
\]

\begin{figure}
	\noindent\framebox[\linewidth]{\parbox{\linewidth-20pt}{
		\begin{center}
			\scalebox{1} % Change this value to rescale the drawing.
			{
				\input ./fig-normal_distance.tex
			}
		\end{center}
		\caption{An example of a subset $A$ of the Euclidean plane $\R^{2}$ for which the normal distance $d_{\perp}(0, A) = 1$ unit (\emph{cf.}\ the dashed line), as opposed to the Euclidean Hausdorff distance $d_{\mathrm{Haus}}(0, A) = 2$ units (\emph{cf.}\ the dotted arc).}
		\label{fig:normal_distance}
	}}
\end{figure}

For any $x \in \mathcal{X}$ and $A \subseteq B \subseteq \mathcal{X}$, it holds that $d_{\perp, \Psi}(x, B) \leq d_{\perp, \Psi}(x, A)$.  Furthermore, since a closed half-space $\H_{p, \nu}$ contains $A$ if, and only if, it contains the closed convex hull $\CCH(A)$ of $A$, the following equality holds:
\begin{equation}
	\label{eq:normal_distance_convex_hull}
	d_{\perp, \Psi}(x, A) = d_{\perp, \Psi}(x, \CCH(A)) \text{ for all } x \in \mathcal{X} \text{ and all } A \subseteq \mathcal{X}.
\end{equation}

\begin{rmk}
	It is natural to ask what, if any, relation there is between the normal distance and Talagrand's convex distance.  The simplest answer is to say that the two distances should be compared only with great caution, since each belongs to a different setting:  Talagrand's distance is defined on a product of sets, whereas the normal distance is defined on a topological vector space.  Even on $\R^{N}$, the two distances measure different quantities:  in some sense, $d_{\mathrm{Tal}}(x, A)$ measures how many of the coordinates of $x$ are covered by $A$, but does not measure the geometric distance between them;  on the other hand, $d_{\perp, \Psi}(x, A)$ is a much more geometric measure of how far $x$ is from $A$ in terms of linear functionals on $\mathcal{X}$, and the ``size'' of those linear functionals is measured by $\Psi$.  In particular, Talagrand's convex distance is positively homogeneous of degree zero, whereas the normal distance is positively homogeneous of degree one:  for any $x \in \R^{N}$, $A \subseteq \R^{N}$, and $\alpha > 0$,
	\[
		d_{\mathrm{Tal}}(\alpha x, \alpha A) = d_{\mathrm{Tal}}(x, A),
	\]
	\[
		d_{\perp, \Psi}(\alpha x, \alpha A) = \alpha d_{\perp, \Psi}(x, A).
	\]
\end{rmk}

This section concludes with the proof of the portmanteau theorem (theorem \ref{thm:portmanteau}) and some final remarks on its applicability:

\begin{proof}[Proof of theorem \ref{thm:portmanteau}.]
	The equivalence will be established by showing that
	\[
		\text{(\ref{port-1})} \implies \text{(\ref{port-2})} \implies \text{(\ref{port-3})} \implies \text{(\ref{port-4})} \implies \text{(\ref{port-5})} \implies \text{(\ref{port-1})}.
	\]

	Suppose that (\ref{port-1}) holds.  Then
	\begin{align*}
		& \P[X \in K] \\
		& \quad \leq \inf_{\H_{p, \nu} \supseteq K} \P[X \in \H_{p, \nu}] & & \text{by monotonicity of $\P$,} \\
		& \quad \leq \inf_{\H_{p, \nu} \supseteq K} \exp \left( - \frac{d_{\perp, \Psi}(\Exp[X], \H_{p, \nu})^{2}}{2} \right) & & \text{by (\ref{port-1}),} \\
		& \quad = \exp \left( - \frac{1}{2} \sup_{\H_{p, \nu} \supseteq K} d_{\perp, \Psi}(\Exp[X], \H_{p, \nu})^{2} \right) & & \\
		& \quad = \exp \left( - \frac{d_{\perp, \Psi}(\Exp[X], K)^{2}}{2} \right) & & \text{by \eqref{eq:normal_distance}.}
	\end{align*}
	Hence, (\ref{port-1}) implies (\ref{port-2}).
	
	Suppose that (\ref{port-2}) holds;  then
	\begin{align*}
		\P[X \in A]
		& \leq \P[X \in \CCH(A)] & & \text{since $A \subseteq \CCH(A)$,} \\
		& \leq \exp \left( - \frac{d_{\perp, \Psi}(\Exp[X], \CCH(A))^{2}}{2} \right) & & \text{by (\ref{port-2}),} \\
		& = \exp \left( - \frac{d_{\perp, \Psi}(\Exp[X], A)^{2}}{2} \right) & & \text{by \eqref{eq:normal_distance_convex_hull},}
	\end{align*}
	and so (\ref{port-2}) implies (\ref{port-3}).  (\ref{port-4}) follows from (\ref{port-3}) upon setting $A := \{ x \in \mathcal{X} \mid f(x) \leq \theta \}$.  (\ref{port-5}) is clearly a special case of (\ref{port-4}).  (\ref{port-1}) follows from (\ref{port-5}) upon setting $f := \chi_{\H_{p, \nu}}$ and $\theta := 1$.
\end{proof}

\begin{rmk}
	It is important to note that all the bounds in theorem \ref{thm:portmanteau} may be trivial if the dual space $\mathcal{X}^{\ast}$ is not rich enough.  For example, given a measure space $(\mathcal{Z}, \mathscr{F}, \mu)$, for $0 < p < 1$, the space
	\[
		\mathcal{L}^{p}(\mathcal{Z}, \mathscr{F}, \mu; \R) := \left\{ f \colon \mathcal{Z} \to \R \smid \| f \|_{p} := \left( \int_{\mathcal{Z}} | f(z) |^{p} \, \d \mu(z) \right)^{1/p} < + \infty \right\}
	\]
	is a topological vector space with respect to the quasinorm topology generated by $\| \cdot \|_{p}$.  This space is not locally convex and has a trivial dual space:  the only continuous linear functional on this space is the zero functional, and so the only closed half-space is the whole space.  See \emph{e.g.}\ \cite[\S1.47]{Rudin:1991} for further discussion of spaces such as $\mathcal{L}^{p}([0, 1]; \R)$ for $0 < p < 1$.
	
	It is tempting to eliminate these pathologies by working with the algebraic, instead of the topological, dual of $\mathcal{X}$.  This can be done, and most results go through \emph{mutatis mutandis};  in particular, it is necessary to replace all references to the closed convex hull $\CCH(A)$ of $A \subseteq \mathcal{X}$ with the convex hull $\CH(A)$;  the analogue of \eqref{eq:normal_distance_convex_hull} (with $\Psi$ now defined on the algebraic dual of $\mathcal{X}$) is
	\[
		d_{\perp, \Psi}(x, A) = d_{\perp, \Psi}(x, \CH(A)) \text{ for all } x \in \mathcal{X} \text{ and all } A \subseteq \mathcal{X}.
	\]
	The principal disadvantage of ignoring all topological structure on $\mathcal{X}$, of course, is that there are no longer notions of interior, closure and frontier --- although it still makes sense to discuss the extremal points of convex sets.
\end{rmk}

\section{Normal Distance as a Concentration Rate}
\label{sec:main}

The method of Chernoff bounding (reviewed in lemma \ref{lem:Chernoff}) gives bounds on $\P[X \in \H_{p, \nu}]$ in terms of the moment-generating function $M_{X}$.  If these bounds can be formulated in terms of a suitable normal distance, then theorem \ref{thm:portmanteau} produces equivalent bounds for on $\P[X \in K]$ for convex $K$, on $\P[X \in A]$, \emph{\& c.}.  As noted in \cite[\S2]{Lugosi:2009}, the best Chernoff bound on $\P[f(X) \geq \theta]$ is never better than the best bound using the all the moments of $f(X)$:  if $f$ takes only non-negative values, then
\begin{equation}
	\label{eq:moments_vs_Chernoff}
	\inf_{k \in \N} \theta^{-k} \Exp \big[ f(X)^{k} \big] \leq \inf_{s \geq 0} e^{-s \theta} \Exp \big[ e^{s f(X)} \big].
\end{equation}
However, Chernoff bounds have the advantage that they are geometrically very easy to handle.

The next result provides the normal distance formulation for an $\mathcal{X}$-valued Gaussian random variable (in fact, for a family of such variables).  In the special case of a single Gaussian random vector $X$ on $\mathcal{X} = \R^{N}$ with covariance operator $C_{X} = \sigma \mathbb{I}_{N}$, proposition \ref{prop:Gaussian} yields the classical Chernoff bound for a multivariate normal random variable.

\begin{prop}
	\label{prop:Gaussian}
	Let $\Gamma$ be a family of Gaussian random vectors in $\mathcal{X}$.  For each $X \in \Gamma$, let $C_{X} \colon \mathcal{X}^{\ast} \to \mathcal{X}^{\ast \ast}$ be its covariance operator defined by
	\begin{equation}
		\label{eq:covariance}
		\langle C_{X} \ell, \nu \rangle := \Exp \left[ \langle \ell, X \rangle \langle \nu, X \rangle \right].
	\end{equation}
	Let $E := \{ \Exp[X] \mid X \in \Gamma \}$, let
	\begin{equation}
		\label{eq:Gaussian_Psi}
		\Psi(\nu) := \sup_{X \in \Gamma} \sqrt{\langle C_{X} \nu, \nu \rangle},
	\end{equation}
	and let $d_{\perp, \Psi}$ be the corresponding normal distance.  Then, for any $A \subseteq \mathcal{X}$,
	\begin{equation}
		\label{eq:Gaussian_bound}
		\sup_{X \in \Gamma} \P[X \in A] \leq \exp \left( - \frac{d_{\perp, \Psi}(E, A)^{2}}{2} \right).
	\end{equation}
\end{prop}

\begin{proof}
	For each $X \in \Gamma$, the moment-generating function for $X$ is given by
	\begin{equation}
		\label{eq:Gaussian_MGF_bound}
		M_{X}(\ell) := \Exp \left[ e^{\langle \ell, X \rangle} \right] = \exp \left( \langle \ell, \Exp[X] \rangle + \frac{\langle C_{X} \ell, \ell \rangle}{2} \right).
	\end{equation}
	Therefore,
	\begin{align*}
		& \P[X \in \H_{p, \nu}] \\
		& \quad \leq \inf_{s \geq 0} \exp \left( s \langle \nu, p - \Exp[X] \rangle + s^{2} \frac{\langle C_{X} \nu, \nu \rangle}{2} \right) & & \text{by \eqref{eq:Gaussian_MGF_bound} and lemma \ref{lem:Chernoff},} \\
		& \quad = \exp \left( - \frac{\langle \nu, \Exp[X] - p \rangle_{+}^{2}}{2 \langle C_{X} \nu, \nu \rangle^{2}} \right) & & \\
		& \quad \leq \exp \left( - \frac{\langle \nu, \Exp[X] - p \rangle_{+}^{2}}{2 \Psi(\nu)^{2}} \right) & & \text{by \eqref{eq:Gaussian_Psi},} \\
		& \quad = \exp \left( - \frac{d_{\perp, \Psi}(\Exp[X], \H_{p, \nu})^{2}}{2} \right) & & \text{by \eqref{eq:normal_distance}.}
	\end{align*}
	Hence, by theorem \ref{thm:portmanteau},
	\[
		\P[X \in A] \leq \exp \left( - \frac{d_{\perp, \Psi}(\Exp[X], A)^{2}}{2} \right),
	\]
	and so
	\begin{align*}
		\sup_{X \in \Gamma} \P[X \in A]
		& \leq \sup_{X \in \Gamma} \exp \left( - \frac{d_{\perp, \Psi}(\Exp[X], A)^{2}}{2} \right) \\
		& = \exp \left( - \inf_{X \in \Gamma} \frac{d_{\perp, \Psi}(\Exp[X], A)^{2}}{2} \right) \\
		& = \exp \left( \frac{d_{\perp, \Psi}(E, A)^{2}}{2} \right). \qedhere
	\end{align*}
\end{proof}

Lemma \ref{lem:Chernoff} also has the following consequences for random vectors supported in a cuboid in $\R^{N}$;  this encompasses two standard situations in which concentration is often studied, namely concentration for functions on the Euclidean unit cube and on the Hamming cube.

\begin{prop}
	\label{prop:cuboid}
	Let $X$ be a random vector in $\R^{N}$ with independent components such that each component $X_{n}$ almost surely takes values in a fixed interval of length $L_{n}$.  Let
	\begin{equation}
		\label{eq:cuboid_Psi}
		\Psi(\nu) := \frac{1}{2} \sqrt{ \sum_{n = 1}^{N} L_{n}^{2} \nu_{n}^{2} }
	\end{equation}
	and let $d_{\perp, \Psi}$ be the corresponding normal distance.  Then, for any $A \subseteq \R^{N}$,
	\begin{equation}
		\label{eq:cuboid_bound}
		\P[X \in A] \leq \exp \left( - \frac{d_{\perp, \Psi}(\Exp[X], A)^{2}}{2} \right).
	\end{equation}
	\emph{A fortiori}, if $X$ takes values in (a translate of) the unit cube $[0, 1]^{N}$, then
	\begin{equation}
		\label{eq:unit_cube_bound}
		\P[X \in A] \leq \exp \left( - 2 d_{\perp}(\Exp[X], A)^{2} \right),
	\end{equation}
	and if $X$ takes values in (a translate of) the Hamming cube $\{ -1, +1 \}^{N}$, then
	\begin{equation}
		\label{eq:hamming_cube_bound}
		\P[X \in A] \leq \exp \left( - \frac{d_{\perp}(\Exp[X], A)^{2}}{2} \right).
	\end{equation}	
\end{prop}

\begin{proof}
	The proof is similar to the Gaussian case:  it is an application of lemma \ref{lem:Chernoff} and Hoeffding's lemma \cite[lemma 1 and (4.16)]{Hoeffding:1963}, which bounds the moment-generating function of $X_{n}$ as follows:
	\[
		M_{X_{n}}(\ell_{n}) := \Exp \left[ \exp(\ell_{n} X_{n}) \right] \leq \exp \left( \ell_{n} \Exp[X_{n}] + \frac{\ell_{n}^{2} L_{n}^{2}}{8} \right).
	\]
	
	Note that the claim can also be proved directly by applying McDiarmid's inequality to the function $\langle \nu, \cdot \rangle$, which has mean $\Exp[ \langle \nu, X \rangle] = \langle \nu, \Exp[X] \rangle$ and McDiarmid diameter $\sqrt{L_{1}^{2} + \dots + L_{N}^{2}}$.
\end{proof}

\begin{rmk}
	Note the similarity between the normal distances of propositions \ref{prop:Gaussian} and \ref{prop:cuboid}.  In the Gaussian case, the norm on $\mathcal{X}^{\ast}$ is the one induced by the ``largest'' covariance operator in the family of random variables $\Gamma$.  In the bounded-range case, the norm on $\mathcal{X}^{\ast}$ is the one induced by the ``largest'' covariance operator for random variables satisfying the range constraint:  if $X$ is a real-valued random variable taking values in an interval $[a, b]$, then $\Psi(\nu)^{2} = \frac{1}{4} (b - a)^{2} \nu^{2}$ and $\mathrm{Var}[X] \leq \frac{1}{4} (b - a)^{2}$; this upper bound on the variance is attained by a Bernoulli random variable with law $\frac{1}{2} \delta_{a} + \frac{1}{2} \delta_{b}$.
\end{rmk}

The next result identifies the normal distance that corresponds to the concentration of a vector, the entries of which are the empirical (sampled) means of functions of independent random variables.

\begin{prop}
	\label{prop:emp_mean}
	For $n = 1, \dots, N$, let $Z_{n} := f_{n}(Y_{n, 1}, \dots, Y_{n, K(n)})$ be a real-valued function of independent random variables $Y_{n, k}$, and suppose that $f_{n}$ has finite McDiarmid diameter $\mathcal{D}[f_{n}]$.  Let $Z = (Z_{1}, \dots, Z_{N})$.  Suppose that the random inputs of each $f_{n}$ are sampled independently $M(n)$ times according to the distribution $\P$ and that the empirical average
	\begin{equation}
		\label{eq:emp_mean}
		\widehat{\Exp}[Z] = \left( \frac{1}{M(n)} \sum_{m = 1}^{M(n)} f_n \left( Y_{n, 1}^{(m)}, \dots, Y_{n, K(n)}^{(m)} \right) \right)_{n = 1}^{N} \in \R^{N}
	\end{equation}
	is formed.  Then, for any $A \subseteq \R^{N}$,
	\begin{equation}
		\label{eq:emp_mean_bound}
		\P \left[ \widehat{\Exp}[Z] \in A \right] \leq \exp \left( - \frac{d_{\perp, \Psi}(\Exp[Z], A)^{2}}{2} \right),
	\end{equation}
	where the distance $\Psi \colon (\R^{N})^{\ast} \to [0, +\infty)$ is given in terms of the McDiarmid diameters of the functions $f_{1}, \dots, f_{N}$ and the sample sizes $M(1), \dots, M(N)$:
	\begin{equation}
		\label{eq:emp_mean_Psi}
		\Psi(\nu) := \frac{1}{2} \left( \sum_{n = 1}^{N} \frac{\nu_{n}^{2} \mathcal{D}[f_{n}]^{2}}{M(n)} \right)^{1/2}.
	\end{equation}
\end{prop}

\begin{proof}
	Let $\H_{p, \nu} \subsetneq \R^{N}$ be a half-space.  Consider the real-valued random variable $\left\langle \nu, \widehat{\Exp}[Z] \right\rangle$ as a function of the sampled input random variables $Y_{n, k}^{(m)}$.  Suppose that the McDiarmid subdiameter of $f_{n}$ with respect to $Y_{n, k}$ is $D_{n, k}$.  Then the McDiarmid subdiameter of $\left\langle \nu, \widehat{\Exp}[Z] \right\rangle$ with respect to the $m^{\mathrm{th}}$ sample of $Y_{n, k}$ is $\nu_{n} D_{n, k} / M(n)$.  Hence, the McDiarmid diameter of $\left\langle \nu, \widehat{\Exp}[Z] \right\rangle$ is
	\[
		\sqrt{ \sum_{k, n, m} \frac{\nu_{n}^{2} D_{n, k}^{2}}{M(n)^{2}} } = \sqrt{ \sum_{n, m} \frac{\nu_{n}^{2} \mathcal{D}[f_{n}]^{2}}{M(n)^{2}} } = \sqrt{ \sum_{n} \frac{\nu_{n}^{2} \mathcal{D}[f_{n}]^{2}}{M(n)} }
	\]
	Therefore, since $\widehat{\Exp}[Z]$ is an unbiased estimator for $\Exp[Z]$ (\emph{i.e.}\ $\Exp \left[ \widehat{\Exp}[Z] \right] = \widehat{\Exp}[Z]$), McDiarmid's inequality \eqref{eq:McD2-leq} implies that
	\begin{align*}
		\P \left[ \widehat{\Exp}[Z] \in \H_{p, \nu} \right]
		& = \P \left[ \left\langle \nu, \widehat{\Exp}[Z] \right\rangle \leq \langle \nu, p \rangle \right] \\
		& \leq \exp \left( - \frac{2 \left( \langle \nu, \Exp[Z] \rangle - \langle \nu, p \rangle \right)_{+}^{2}}{ \sum_{n = 1}^{N} \frac{\nu_{n}^{2} \mathcal{D}[f_{n}]^{2}}{M(n)} } \right) \\
		& = \exp \left( - \frac{\langle \nu, \Exp[Z] - p \rangle_{+}^{2}}{2 \cdot \frac{1}{4} \cdot \sum_{n = 1}^{N} \frac{\nu_{n}^{2} \mathcal{D}[f_{n}]^{2}}{M(n)}} \right) \\
		& = \exp \left( - \frac{d_{\perp, \Psi}(\Exp[Z], \H_{p, \nu})^{2}}{2} \right).
	\end{align*}
	The claim now follows from theorem \ref{thm:portmanteau}.
\end{proof}

An example of the application of proposition \ref{prop:emp_mean} is the following:

\begin{eg}[Functions of empirical means]
	The Chernoff bounding method can be used to provide much-improved confidence levels for quantities derived from many empirical --- as opposed to exact --- means;  see \emph{e.g.}\ \cite[\S5]{SullivanTopcuMcKernsOwhadi:2010}.  Suppose that $H_{0} \colon \R^{N} \to \R$ is some function of interest:  in particular, the quantity of interest is $H_{0} \left( \Exp[Z_{1}], \dots, \Exp[Z_{N}] \right)$ for some absolutely integrable real-valued random variables $Z_{1}, \dots, Z_{N}$.  If, however, the exact means $\Exp[Z_{n}]$ are unknown, then empirical means $\widehat{\Exp}[Z_{n}]$ may be used in their place if appropriate confidence corrections are made.  Suppose that ``error'' corresponds to concluding, based on the empirical means, that $H_{0}(\Exp[Z])$ is smaller than it actually is.  Given $\alpha \in \R^{N}$, set
	\begin{equation}
		\label{eq:H_alpha}
		H_{\alpha}(z_{1}, \dots, z_{N}) := H_{0}(z_{1} + \alpha_{1}, \dots, z_{N} + \alpha_{N}).
	\end{equation}
	Therefore, given any $\eps > 0$, we seek an appropriate ``margin hit'' $\alpha = \alpha(\eps) \in \R^{N}$ (typically, $\alpha_{n} \geq 0$ for each $n \in \{ 1, \dots, N \}$) such that
	\[
		\P \left[ H_{\alpha} \left( \widehat{\Exp}[Z_{1}], \dots, \widehat{\Exp}[Z_{N}] \right) \geq H_{0} \left( \Exp[Z_{1}], \dots, \Exp[Z_{N}] \right) \right] \geq 1 - \eps.
	\]
	Dually, given $\alpha \in \R^{N}$, we seek a sharp upper bound on the probability of error, \emph{i.e.}\ on
	\[
		\P \left[ H_{\alpha} \left( \widehat{\Exp}[Z_{1}], \dots, \widehat{\Exp}[Z_{N}] \right) \leq H_{0} \left( \Exp[Z_{1}], \dots, \Exp[Z_{N}] \right) \right].
	\]
		
	If $H_{0}$ (and hence $H_{\alpha}$) is monotonic in each of its $N$ arguments and $Z_{1}, \dots, Z_{N}$ are independent, then the probability of non-error can be bounded from below as follows:
	\begin{align*}
		\P \left[ H_{\alpha} \left( \widehat{\Exp}[Z] \right) \leq H_{0}(\Exp[Z]) \right]
		& = \P \left[ H_{\alpha} \left( \widehat{\Exp}[Z] \right) \leq H_{\alpha}(\Exp[Z] - \alpha) \right] \\
		& \leq \prod_{n = 1}^{N} \P \left[ \widehat{\Exp}[Z_{n}] \leq \Exp[Z_{n}] - \alpha_{n} \right] \\
		& \leq 1 - \prod_{n = 1}^{N} \left( 1 - \exp \left( - \frac{- 2 M(n) (\alpha_{n})_{+}^{2}}{\mathcal{D}[f_{n}]^{2}} \right) \right).
	\end{align*}
	Unfortunately, when $N$ is large, the last line of this inequality is typically close to zero unless the sample sizes are very large, and so this bound is of limited use.  Geometrically, this is analogous to the fact that a high-dimensional orthant (product of half-lines) appears to be very narrow from the perspective of an observer at its vertex.  In contrast, half-spaces always fill a half of the observer's field of view.  To bound the probability of sublevel or superlevel sets using half-spaces requires $H_{\alpha}$ to have some convexity --- not monotonicity --- properties.
	
	If $H_{\alpha}$ is quasiconvex, then the bounds using normal distances can be applied to good effect, and yield estimates that actually perform better the larger $N$ is.  In particular, if $H_{\alpha}$ is both quasiconvex and differentiable, then the outward normal to its $t$-level set at some point $p$ is just any positive multiple of the derivative of $H_{\alpha}$ at $p$, and this yields the bound
	\begin{equation}
		\label{eq:pre_link_formula}
		\P \left[ H_{\alpha} \left( \widehat{\Exp}[Z] \right) \leq \theta \right] \leq \inf_{p : H_{\alpha}(p) \leq \theta} \exp \left( - \frac{2 \left( \sum_{n = 1}^{N} \partial_{n} H_{\alpha}(p) (\Exp[Z_{n}] - p_{n}) \right)_{+}^{2}}{ \sum_{n = 1}^{N} \frac{(\partial_{n} H_{\alpha}(p))^{2} \mathcal{D}[f_{n}]^{2}}{M(n)} } \right).
	\end{equation}
	In particular, taking $\theta = H_{0}(\Exp[Z]) = H_{\alpha}(\Exp[Z] - \alpha)$ and evaluating the exponential in \eqref{eq:pre_link_formula} at $p = \Exp[Z] - \alpha \in \R^{N}$ yields that
	\begin{equation}
		\label{eq:link_formula}
		\P \left[ H_{\alpha} \left( \widehat{\Exp}[Z] \right) \leq H_{0}(\Exp[Z]) \right] \leq \exp \left( - \frac{2 \left( \sum_{n = 1}^{N} \partial_{n} H_{\alpha}(p) \alpha_{n} \right)_{+}^{2}}{ \sum_{n = 1}^{N} \frac{(\partial_{n} H_{\alpha}(p))^{2} \mathcal{D}[f_{n}]^{2}}{M(n)} } \right).
	\end{equation}
	\eqref{eq:link_formula} is particularly useful since it links the margin hits $\alpha_{n}$, the sample sizes $M(n)$, and the maximum probability of error.  For example, given a desired level of confidence, margin hits $\alpha_{n}$, and a total number of samples $M \in \N$, one can choose sample sizes $M(1), \dots, M(N)$ that sum to $M$ and minimize the right-hand side of \eqref{eq:link_formula};  this yields an optimal distribution of sampling resources so as to ensure that $H_{\alpha} \left( \widehat{\Exp}[Z] \right) \geq H_{0}(\Exp[Z])$ with the desired level of confidence.
\end{eg}

\section{High-Dimensional Asymptotics}
\label{sec:asymptotics}

The topic of this section is the asymptotic sharpness of the bounds introduced above as the dimension of the space $\mathcal{X}$ becomes large.  We begin with a comparison of the McDiarmid and half-space bounds for a simple function:  a quadratic form on $\R^{N}$.

\begin{eg}[Comparison with McDiarmid's inequality]
	\label{eg:quadratic_form_comparison}
	The following example serves to illustrate how the half-space method can produce upper bounds on the measure of suitable sublevel sets that are superior to those offered by McDiarmid's inequality;  it also shows how this effect is more pronounced in higher-dimensional spaces.  Consider the following quadratic form $Q_{N}$ on $\R^{N}$:
	\begin{equation}
		\label{eq:quadratic_form}
		Q_{N}(x) := \tfrac{1}{2} \left\| x - \left( \tfrac{1}{2}, \dots, \tfrac{1}{2} \right) \right\|_{2}^{2}.
	\end{equation}
	For any $\theta > 0$, the sublevel set $Q_{N}^{-1}([-\infty, \theta])$ is simply a ball of radius $\sqrt{2 \theta}$ about the point $\left( \tfrac{1}{2}, \dots, \tfrac{1}{2} \right)$.  Suppose that a random vector $X$ takes values in $\left[ - \tfrac{1}{2}, + \tfrac{1}{2} \right]^{N}$ with independent components.  McDiarmid's inequality \eqref{eq:McD2-leq} implies that
	\[
		\P[Q_{N}(X) \leq \theta] \leq \exp \left( - 8 \left( \frac{\sqrt{N}}{6} - \frac{\theta}{\sqrt{N}} \right)_{+}^{2} \right),
	\]
	If also $\Exp[X] = 0$, then corollary \ref{prop:cuboid} implies that
	\[
		\P[Q_{N}(X) \leq \theta] \leq \exp \left( - \frac{( \sqrt{N} - \sqrt{8 \theta} )_{+}^{2}}{2} \right).
	\]
	For small $N$ and large $\theta$, McDiarmid's bound is the sharper of the two.  However, for small $\theta$ (and, notably, as $N \to \infty$ for any fixed $\theta$), the half-space bound is the sharper bound.  See figure \ref{fig:quadratic_form_comparison} for an illustration.

	\begin{figure}
		\noindent\framebox[\linewidth]{\parbox{\linewidth-20pt}{
			\begin{center}
				\scalebox{1} % Change this value to rescale the drawing.
				{
					\input ./fig-quadratic_form_comparison.tex
				}
			\end{center}
			\caption{For the quadratic form $Q_{N}$ on $\R^{N}$ given in \eqref{eq:quadratic_form}, a comparison of the McDiarmid upper bound (squares) and half-space upper bound (triangles) on $\P[Q_{N}(X) \leq \theta]$ in the cases $\theta = \tfrac{1}{4}$ (dotted line and hollow polygons) and $\theta = \tfrac{1}{8}$ (solid line and filled polygons).}
			\label{fig:quadratic_form_comparison}
		}}
	\end{figure}
\end{eg}

The previous example suggests that bounds constructed using the half-space method may perform very well in high dimension but also that the sharpness of the bound may depend on ``how round'' the set whose measure we wish to bound is.  To fix ideas, suppose that $X = (X_{1}, \dots, X_{N}) \colon \Omega \to \R^{N}$ is a random vector with independent components, where $X_{n}$ is supported on an interval of length $L_{n}$.  For $A \subseteq \R^{N}$, how sharp is the bound
\begin{equation}
	\label{eq:bound_for_asymptotics}
	\P[X \in A] \leq \exp \left( - \frac{d_{\perp}(\Exp[X], A)^{2}}{2} \right)?
\end{equation}
First, note that since $d_{\perp}(\Exp[X], A) = d_{\perp}(\Exp[X], \CCH(A))$, the bound cannot be expected to be sharp if $A$ differs greatly from its closed convex hull, and so it makes sense to restrict investigation to the case that $A = K$, a closed and convex subset of $\R^{N}$.  Secondly, it is not reasonable to expect the bound \eqref{eq:bound_for_asymptotics} on $\P[X \in K]$ to be sharp if $K$ is sharply pointed, \emph{e.g.}\ if $K$ is the narrow wedge $K_{\eps}$ of angle $\eps \ll 1$ based at $e_{1} := (1, 0, \dots, 0)$ in $\R^{N}$:
\begin{equation}
	\label{eq:narrow_wedge}
	K_{\eps} := \left\{ x \in \R^{N} \smid \frac{(x - e_{1}) \cdot e_{1}}{\| x - e_{1} \|_{2}} \leq \eps \right\};
\end{equation}
see figure \ref{fig:narrow_wedge}.  Therefore, we wish to consider the opposite situation in which $K$ has no sharp points, which will be made precise by requiring that $K$ satisfy an interior ball condition.

\begin{figure}
	\noindent\framebox[\linewidth]{\parbox{\linewidth-20pt}{
		\begin{center}
			\scalebox{1} % Change this value to rescale the drawing.
			{
				\input ./fig-narrow_wedge.tex
			}
		\end{center}
		\caption{It is not reasonable to expect that (an upper bound for) the measure of the half-space $\H_{e_{1}, -e_{1}}$ is a sharp upper bound for the measure of the narrow wedge $K_{\eps}$ when $\eps$ is small.}
		\label{fig:narrow_wedge}
	}}
\end{figure}

Suppose that $(p, \nu) \in \mathrm{N}^{\ast} K$ is such that $d_{\perp}(x, \H_{p, \nu}) = d_{\perp}(x, K)$.  Suppose also that $\Ball_{r}(p - r \omega) \subseteq K$, with $r > 0$ and $\omega \in \R^{N}$ a unit vector, is an interior ball for $K$ at $p \in \partial K$;  \emph{cf.}\ figure \ref{fig:interior_ball}.  If the law of $X$ on $\R^{N}$ is highly singular, then it cannot be expected that the bound \eqref{eq:bound_for_asymptotics} is sharp, so suppose that the law of $X$ has a density with respect to Lebesgue measure that is bounded above by some constant $C > 0$.  Then the bound \eqref{eq:bound_for_asymptotics} is
\[
	\P[X \in K] \leq \exp \left( - \frac{2 \langle \nu, \Exp[X] - p \rangle_{+}^{2}}{\sum_{n = 1}^{N} \nu_{n}^{2} L_{n}^{2}} \right).
\]
In the extreme case, $K$ is precisely the closed ball $\overline{\Ball}_{r}(p - r \omega)$, the $\P$-measure of which is at most $C r^{N} \pi^{N / 2} / \Gamma (1 + N/2)$.

\begin{figure}
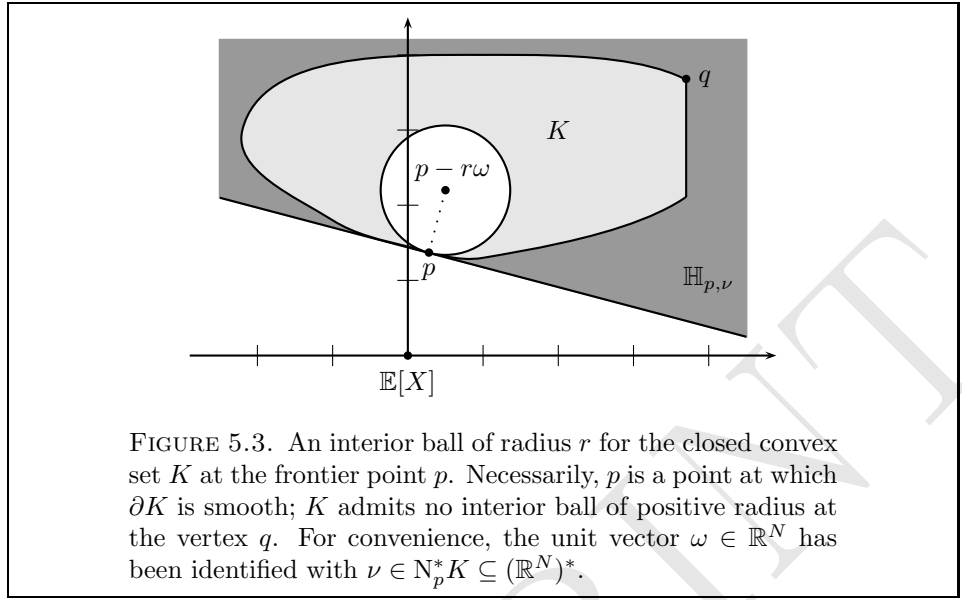

	\noindent\framebox[\linewidth]{\parbox{\linewidth-20pt}{
		\begin{center}
			\scalebox{1} % Change this value to rescale the drawing.
			{
				\input ./fig-interior_ball.tex
			}
		\end{center}
		\caption{An interior ball of radius $r$ for the closed convex set $K$ at the frontier point $p$.  Necessarily, $p$ is a point at which $\partial K$ is smooth;  $K$ admits no interior ball of positive radius at the vertex $q$.  For convenience, the unit vector $\omega \in \R^{N}$ has been identified with $\nu \in \mathrm{N}_{p}^{\ast} K \subseteq (\R^{N})^{\ast}$.}
		\label{fig:interior_ball}
	}}
\end{figure}

In large deviations theory, the standard notion of asymptotic sharpness is logarithmic equivalence \cite[\S{}I.1]{denHollander:2000};  see also \emph{e.g.}\ \cite{DemboZeitouni:1998} \cite{Varadhan:2008} for surveys of the large deviations literature.  Two sequences $(\alpha_{n})_{n \in \N}$ and $(\beta_{n})_{n \in \N}$ are said to be \emph{logarithmically equivalent}, denoted $\alpha_{n} \simeq \beta_{n}$, if
\begin{equation}
	\label{eq:logarithmic_equivalence}
	\frac{1}{n} \log \alpha_{n} - \frac{1}{n} \log \beta_{n} \equiv \log \left( \frac{\alpha_{n}}{\beta_{n}} \right)^{1/n} \to 0 \text{ as } n \to \infty.
\end{equation}
Are the half-space bound \eqref{eq:bound_for_asymptotics} and the measure of $\overline{\Ball}_{r}(p - r \omega)$ logarithmically equivalent?  That is, does the conditional probability ${ \P \left[ X \in \overline{\Ball}_{r}(p - r \omega) \smid X \in \H_{p, \nu} \right] }$, when raised to the power $\frac{1}{N}$, converge to $1$ as $N \to \infty$?  To simplify the asymptotic expansions below, in all lines after the first two, we shall take $\Exp[X] = 0$ and $L_{1} = \dots = L_{N} = 1$.  Then
\begin{align*}
	& \frac{1}{N} \log \P \left[X \in \overline{\Ball}_{r}(p - r \omega) \right] - \frac{1}{N} \log \left( \text{r.h.s.\ of \eqref{eq:bound_for_asymptotics}} \right)\\
	& \quad \leq \frac{1}{N} \left( \log \frac{C r^{N} \pi^{N / 2}}{\Gamma (1 + N/2)} + \frac{2 \langle \nu, \Exp[X] - p \rangle_{+}^{2}}{\sum_{n = 1}^{N} \nu_{n}^{2} L_{n}^{2}} \right) \\
	& \quad = \frac{2 \langle \nu, p \rangle_{-}^{2}}{N \| \nu \|_{2}^{2}} + \frac{\log (C r^{N} \pi^{N/2})}{N} - \frac{\log \Gamma (1 + N/2)}{N} \\
	\intertext{which, by Stirling's approximation for the Gamma function \cite[p.\ 256, eq.\ (6.1.37)]{AbramowitzStegun:1992}, is approximately}
	& \quad \approx \frac{2 \langle \nu, p \rangle_{-}^{2}}{N \| \nu \|_{2}^{2}} + \frac{\log (C r^{N} \pi^{N/2})}{N} - \frac{1}{N} \log \left( \sqrt{\frac{2 \pi}{1 + N/2}} \left( \frac{1 + N/2}{e} \right)^{1 + N/2} \right) \\
	& \quad \sim \frac{2 \langle \nu, p \rangle_{-}^{2}}{N \| \nu \|_{2}^{2}} + \frac{\log C}{N} - \frac{1}{2 N} \log \frac{4 \pi}{N} - \frac{1 + N/2}{N} \log \frac{N}{2 e} \\
	& \quad \sim \frac{2 \langle \nu, p \rangle_{-}^{2}}{N \| \nu \|_{2}^{2}} + \log r - \log \sqrt{N}
\end{align*}
Note that $\langle \nu, p \rangle_{-} / \| \nu \|_{2} \leq \sqrt{N} d_{\mathbf{1}}(0, p)$, where $d_{\mathbf{1}}$ denotes the weighted Hamming distance with weight $w = (1, \dots, 1)$.  Therefore, a necessary (but not sufficient) condition for the half-space bound to be asymptotically sharp in the sense of logarithmic equivalence is that $r$ is of the same order as $\sqrt{N}$.  That is, it is necessary that $K$ is sufficiently round that it has an interior ball of radius comparable to $\sqrt{N}$ at those frontier points where the normal distance $d_{\perp}(\Exp[X], K)$ is attained.

Now suppose that $K = f^{-1}([- \infty, \theta])$ is a convex sublevel set for twice-differentiable function $f$.  Let $\eta_{1}, \dots, \eta_{N - 1}, \nu$ be a basis of $\R^{N}$ such that
\[
	\| \eta_{1} \|_{2} = \dots = \| \eta_{N - 1} \|_{2} = \| \nu \|_{2} = 1
\]
and, for each $n \in \{ 1, \dots, N - 1 \}$, $\eta_{n}$ is perpendicular to $\nu$.  Suppose that, in this system of normal coordinates, near $p$, the frontier of $K$ can be approximated by a parabola:
\[
	\partial K = \left\{ y_{1} \eta_{1} + \dots y_{N - 1} \eta_{N - 1} - y_{N} \nu \smid y_{N} = \sum_{n = 1}^{N - 1} \lambda_{n} y_{n}^{2} \right\}
\]
with $\lambda_{1} \geq \lambda_{2} \geq \dots \geq \lambda_{N - 1} \geq 0$.  Then the condition that $K$ admits an interior ball of radius $r$ at $p$ is the inequality
\[
	r - \sqrt{r^{2} - \sum_{n = 1}^{N - 1} y_{n}^{2}} \geq \sum_{n = 1}^{N - 1} \lambda_{n} y_{n}^{2} \text{ whenever } \sum_{n = 1}^{N - 1} y_{n}^{2} \leq r^2.
\]
This, in turn, leads to the following condition on $\lambda_{1}$:  it must hold that $\lambda_{1} \leq \frac{1}{2 r}$.  Put another way, the half-space method cannot be expected to provide asymptotically sharp bounds for $\P[f(X) \leq \theta]$ if, when $f$ is approximated in normal coordinates near the closest point of $f^{-1}([- \infty, \theta])$ to $\Exp[X]$ by a non-negative quadratic form, that quadratic form has an eigenvalue greater than $(4 N)^{-1/2}$.

\section{Appendix:  Chernoff Bounds}
\label{sec:Chernoff}

The method of Chernoff bounds \cite[\S7.4.3]{BoydVandenberghe:2004} \cite{Chernoff:1952} is a simple one in which the probability of a subset of $\mathcal{X}$ is bounded by that of a containing half-space, and the probability of that half-space is bounded using the moment-generating function of the probability measure.

\begin{lem}[Chernoff bounds]
	\label{lem:Chernoff}
	For any half-space $\H_{p, \nu} \subseteq \mathcal{X}$,
	\begin{equation}
		\label{eq:Chernoff_halfspace}
		\P[X \in \H_{p, \nu}] \leq \inf_{s \geq 0} e^{s \langle \nu, p \rangle} M_{X} (- s \nu).
	\end{equation}
	For any convex set $K \subseteq \mathcal{X}$,
	\begin{subequations}
		\begin{align}
			\P[X \in K]
			\label{eq:Chernoff_convex-1}
			& \leq \inf_{(p, \nu) \in \mathrm{N}^{\ast} K} e^{\langle \nu, p \rangle} M_{X} (- \nu) \\
			\label{eq:Chernoff_convex-2}
			& = \exp \left( - \sup_{p \in K} (\Lambda_{X} + \chi_{-\mathrm{N}_{p}^{\ast} K})^{\star}(p) \right).
		\end{align}
	\end{subequations}
	In particular, for any $x \in \mathcal{X}$,
	\begin{equation}
		\label{eq:Chernoff_convex-3}
		\P[X = x] \leq \exp ( - \Lambda_{X}^{\star}(x) ).
	\end{equation}
\end{lem}

\begin{proof}
	By the definition of the half-space $\H_{p, \nu}$,
	\begin{align*}
		\P \left[ X \in \H_{p, \nu} \right]
		& = \P \left[ \langle \nu, X \rangle \leq \langle \nu, p \rangle \right] \\
		& = \Exp \left[ \one_{[ \langle \nu, p - X \rangle \geq 0 ]} \right] \\
		& \leq \Exp \left[ e^{ s \langle \nu, p - X \rangle } \right] & & \text{for any $s \geq 0$,}\\
		& = e^{s \langle \nu, p \rangle} \Exp \left[ e^{\langle - s \nu, X \rangle} \right] \\
		& \leq e^{s \langle \nu, p \rangle} M_{X} (- s \nu).
	\end{align*}
	Since this inequality holds for any $s \geq 0$, taking the infimum over all such $s$ yields \eqref{eq:Chernoff_halfspace}.  Recall that the outward normal cone to a convex set at any point is closed under multiplication by non-negative scalars;  hence, for any convex set $K \subseteq \mathcal{X}$, taking the infimum of the right-hand side of \eqref{eq:Chernoff_halfspace} over half-spaces $\H_{p, \nu}$ that contain $K$ yields \eqref{eq:Chernoff_convex-1}.  Now observe that
	\begin{align*}
		& \inf_{(p, \nu) \in \mathrm{N}^{\ast} K} e^{\langle \nu, p \rangle} M_{X} (- \nu) \\
		& \quad = \inf_{(p, \nu) \in \mathrm{N}^{\ast} K} \exp ( \langle \nu, p \rangle + \Lambda_{X} (- \nu) ) \\
		& \quad = \exp \left( \inf_{p \in K} \inf_{\nu \in \mathrm{N}_{p}^{\ast} K} \left( \langle \nu, p \rangle + \Lambda_{X} (- \nu) \right) \right) \\
		& \quad = \exp \left( - \sup_{p \in K} \sup_{\nu \in - \mathrm{N}_{p}^{\ast} K} \left( \langle \nu, p \rangle - \Lambda_{X} (\nu) \right) \right) \\
		& \quad = \exp \left( - \sup_{p \in K} (\Lambda_{X} + \chi_{\vphantom{\tfrac{\sim}{\sim}} -\mathrm{N}_{p}^{\ast} K})^{\star}(p) \right),
	\end{align*}
	which establishes \eqref{eq:Chernoff_convex-2};  \eqref{eq:Chernoff_convex-3} follows as a special case.
\end{proof}

%GATHER{./refs.bib}
\bibliographystyle{amsplain}
\bibliography{./refs}

\end{document}

%% file: fig-normal_cones.tex
\begin{pspicture}(0,-2.1)(6.0,2.1)
	\definecolor{medgray}{rgb}{0.6,0.6,0.6}
	\definecolor{lightgray}{rgb}{0.9,0.9,0.9}
	\pspolygon[linecolor=medgray,fillstyle=solid,fillcolor=medgray](6.0,1.95)(6.0,0.95)(4.0,-0.05)(4.5,1.95)
	\psline[linecolor=medgray,fillcolor=medgray](2.0,0.45)(2.4,2.05)
	\pspolygon[fillstyle=solid,fillcolor=lightgray](0.0,0.95)(0.0,-2.05)(5.0,-2.05)(4.0,-0.05)
	\psdots[dotsize=0 4](2.0,0.45)
	\psdots[dotsize=0 4](4.0,-0.05)
	\rput(1,-1.25){$K$}
	\rput(1.88,0.16){$p$}
	\rput(3.78,-0.24){$q$}
	\rput(1.7,1.5){$\mathrm{N}_{p}^{\ast} K$}
	\rput(5.25,1.5){$\mathrm{N}_{q}^{\ast} K$}
	\psdots[dotsize=0 4](2.5,-1)
	\rput(2.5,-1.25){$r$}
\end{pspicture} 

%% file: fig-normal_distance.tex
\begin{pspicture}(-4,0)(4,4.5)
	\definecolor{medgray}{rgb}{0.6,0.6,0.6}
	\definecolor{lightgray}{rgb}{0.9,0.9,0.9}
	\psccurve[fillstyle=solid,fillcolor=lightgray](-3.5,3)(-2.5,1)(-1,3)(0,2)(1,3)(2.5,1)(3.5,3)(0,4)
	\rput(2.5,2.5){$A$}
	\psaxes[labels=none]{->}(0,0)(-3.9,0)(3.9,4.5)
	\rput(0,-0.25){$0$}
	\psarc[linestyle=dotted](0,0){2}{0}{180}
	\psline{<->}(0,0)(1,1.73205)
	\rput[l](0.5,0.5){$d_{\mathrm{Haus}}(0, A)$}
	\psline[linestyle=dashed](-3.9,1)(3.9,1)
	\psline{<->}(3.2,0)(3.2,1)
	\rput[l](3.45,0.5){$d_{\perp}(0, A)$}
\end{pspicture}

%% file: fig-quadratic_form_comparison.tex
\begin{pspicture}(-1.5,-5.0)(8.5,0.5)
	\psaxes[dx=1,Dx=2]{->}(0,0)(7.75,-4.5)
	\rput(8.25,0.0){$N$}
	\rput{90}(-1.25,-2.5){$\log (\text{upper bound})$}
	% theta = 1/4
	% log(Half-space bound)
	\psline[linestyle=dotted, showpoints=true, dotstyle=triangle, dotsize=0 6](0.5,0.0)(1.0,0.0)(1.5,-0.050510)(2.0,-0.171573)(2.5,-0.337722)(3.0,-0.535898)(3.5,-0.758343)(4.0,-1.000000)(4.5,-1.257359)(5.0,-1.527864)(5.5,-1.809584)(6.0,-2.101021)(6.5,-2.400980)(7.0,-2.708497)
	% log(McDiarmid bound)
	\psline[linestyle=dotted, showpoints=true, dotstyle=square, dotsize=0 6](0.5,0.0)(1.0,-0.027778)(1.5,-0.166667)(2.0,-0.347222)(2.5,-0.544444)(3.0,-0.750000)(3.5,-0.960317)(4.0,-1.173611)(4.5,-1.388889)(5.0,-1.605556)(5.5,-1.823232)(6.0,-2.041667)(6.5,-2.260684)(7.0,-2.480159)
	% theta = 1/8
	% log(Half-space bound)
	\psline[showpoints=true, dotstyle=triangle*, dotsize=0 6](0.5,0.0)(1.0,-0.085786)(1.5,-0.267949)(2.0,-0.500000)(2.5,-0.763932)(3.0,-1.050510)(3.5,-1.354249)(4.0,-1.671573)(4.5,-2.000000)(5.0,-2.337722)(5.5,-2.683375)(6.0,-3.035898)(6.5,-3.394449)(7.0,-3.758343)
	% log(McDiarmid bound)
	\psline[showpoints=true, dotstyle=square*, dotsize=0 6](0.5,-0.013889)(1.0,-0.173611)(1.5,-0.375000)(2.0,-0.586806)(2.5,-0.802778)(3.0,-1.020833)(3.5,-1.240079)(4.0,-1.460069)(4.5,-1.680556)(5.0,-1.901389)(5.5,-2.122475)(6.0,-2.343750)(6.5,-2.565171)(7.0,-2.786706)
\end{pspicture}

%% file: fig-narrow_wedge.tex
\begin{pspicture}(-0.5,-2)(6.5,2)
	\definecolor{medgray}{rgb}{0.6,0.6,0.6}
	\definecolor{lightgray}{rgb}{0.9,0.9,0.9}
	\psframe[dimen=outer,fillstyle=solid,fillcolor=medgray,linecolor=medgray](1,-1.6)(5.6,1.6)
	\psline(1,-1.6)(1,1.6)
	\rput(3,1.2){$\H_{e_{1}, -e_{1}}$}
	\pspolygon[dimen=outer,fillstyle=solid,fillcolor=lightgray,linecolor=lightgray](1,0)(5.6,-1)(5.6,1)
	\psline(5.6,1)(1,0)(5.6,-1)
	\rput(4.5,0.4){$K_{\eps}$}
	\psarc{<->}(1,0){4.2}{-12.5}{12.5}
	\rput[l](5.3,0.5){$\eps$}
	\psaxes[labels=none]{->}(0,0)(0,-1.9)(5.9,1.9)
	\psdot[dotsize=0 4](0,0)
	\rput[r](-0.2,0){$\Exp[X]$}
\end{pspicture}

%% file: fig-interior_ball.tex
\begin{pspicture}(-3,-0.5)(5,4.5)
	\definecolor{medgray}{rgb}{0.6,0.6,0.6}
	\definecolor{lightgray}{rgb}{0.9,0.9,0.9}
	\pspolygon[fillstyle=solid,fillcolor=medgray,linecolor=medgray](-2.5,2.1)(4.5,0.25)(4.5,4.2)(-2.5,4.2)
	\psline(-2.5,2.1)(4.5,0.25)
	\rput(4,1){$\H_{p, \nu}$}
	\psccurve[fillstyle=solid,fillcolor=lightgray](-2.2,3)(-1,1.8)(0,1.45)(1,1.3)(4.2,3)(1,4)
	\rput(2,3){$K$}
	\psframe[fillstyle=solid,fillcolor=medgray,linecolor=medgray](3.7,1.5)(4.4,4)
	\psline{cc-cc}(3.7,2.1)(3.7,3.68)
	\psdot[dotsize=0 4](3.7,3.68)
	\rput(3.95,3.68){$q$}
	\pscircle[fillstyle=solid,fillcolor=white](0.5,2.2){0.875}
	\psdot[dotsize=0 4](0.5,2.2)
	\rput(0.6,2.45){$p - r \omega$}
	\psline[linestyle=dotted](0.5,2.2)(0.28,1.37)
	\psdot[dotsize=0 4](0.28,1.37)
	\rput(0.28,1.12){$p$}
	\psaxes[labels=none]{->}(0,0)(-2.9,0)(4.9,4.5)
	\psdot[dotsize=0 4](0,0)
	\rput[t](0,-0.2){$\Exp[X]$}
\end{pspicture}

%% file: 2010-SO-equiv_conc_ineq.bbl
\def\cprime{$'$}
\providecommand{\bysame}{\leavevmode\hbox to3em{\hrulefill}\thinspace}
\providecommand{\MR}{\relax\ifhmode\unskip\space\fi MR }
\providecommand{\MRhref}[2]{%
  \href{http://www.ams.org/mathscinet-getitem?mr=#1}{#2}
}
\providecommand{\href}[2]{#2}
\begin{thebibliography}{10}

\bibitem{AbramowitzStegun:1992}
M.~Abramowitz and I.~A. Stegun (eds.), \emph{Handbook of {M}athematical
  {F}unctions with {F}ormulas, {G}raphs, and {M}athematical {T}ables}, Dover
  Publications Inc., New York, 1992, Reprint of the 1972 edition. \MR{1225604
  (94b:00012)}

\bibitem{BakryEmery:1985}
D.~Bakry and M.~{\'E}mery, \emph{Diffusions hypercontractives}, S\'eminaire de
  Probabilit\'es, {XIX}, 1983/84, Lecture Notes in Math., vol. 1123, Springer,
  Berlin, 1985, \url{http://dx.doi.org/10.1007/BFb0075847}, pp.~177--206.
  \MR{889476 (88j:60131)}

\bibitem{BobkovLedoux:1998}
S.~G. Bobkov and M.~Ledoux, \emph{On modified logarithmic {S}obolev
  inequalities for {B}ernoulli and {P}oisson measures}, J. Funct. Anal.
  \textbf{156} (1998), no.~2, 347--365,
  \url{http://dx.doi.org/10.1006/jfan.1997.3187}. \MR{1636948 (99e:60051)}

\bibitem{BoucheronLugosiMassart:2003}
S.~Boucheron, G.~Lugosi, and P.~Massart, \emph{Concentration inequalities using
  the entropy method}, Ann. Probab. \textbf{31} (2003), no.~3, 1583--1614,
  \url{http://dx.doi.org/10.1214/aop/1055425791}. \MR{1989444 (2004i:60023)}

\bibitem{BoydVandenberghe:2004}
S.~Boyd and L.~Vandenberghe, \emph{Convex {O}ptimization}, Cambridge University
  Press, Cambridge, 2004. \MR{2061575 (2005d:90002)}

\bibitem{Chernoff:1952}
H.~Chernoff, \emph{A measure of asymptotic efficiency for tests of a hypothesis
  based on the sum of observations}, Ann. Math. Statistics \textbf{23} (1952),
  493--507, \url{http://dx.doi.org/10.1214/aoms/1177729330}. \MR{0057518
  (15,241c)}

\bibitem{Dembo:1997}
A.~Dembo, \emph{Information inequalities and concentration of measure}, Ann.
  Probab. \textbf{25} (1997), no.~2, 927--939,
  \url{http://dx.doi.org/10.1214/aop/1024404424}. \MR{1434131 (98e:60027)}

\bibitem{DemboZeitouni:1998}
A.~Dembo and O.~Zeitouni, \emph{Large {D}eviations {T}echniques and
  {A}pplications}, second ed., Applications of Mathematics (New York), vol.~38,
  Springer-Verlag, New York, 1998. \MR{1619036 (99d:60030)}

\bibitem{denHollander:2000}
F.~den Hollander, \emph{Large {D}eviations}, Fields Institute Monographs,
  vol.~14, American Mathematical Society, Providence, RI, 2000. \MR{1739680
  (2001f:60028)}

\bibitem{Gross:1975}
L.~Gross, \emph{Logarithmic {S}obolev inequalities}, Amer. J. Math. \textbf{97}
  (1975), no.~4, 1061--1083, \url{http://dx.doi.org/10.2307/2373688}.
  \MR{0420249 (54 \#8263)}

\bibitem{Hoeffding:1963}
W.~Hoeffding, \emph{Probability inequalities for sums of bounded random
  variables}, J. Amer. Statist. Assoc. \textbf{58} (1963), no.~301, 13--30,
  \url{http://dx.doi.org/10.2307/2282952}. \MR{0144363 (26 \#1908)}

\bibitem{HolleyStroock:1987}
R.~Holley and D.~Stroock, \emph{Logarithmic {S}obolev inequalities and
  stochastic {I}sing models}, J. Statist. Phys. \textbf{46} (1987), no.~5-6,
  1159--1194, \url{http://dx.doi.org/10.1007/BF01011161}. \MR{893137
  (89e:82013)}

\bibitem{JohnsonSchechtman:1991}
W.~B. Johnson and G.~Schechtman, \emph{Remarks on {T}alagrand's deviation
  inequality for {R}ademacher functions}, Functional {A}nalysis ({A}ustin,
  {TX}, 1987/1989), Lecture Notes in Math., vol. 1470, Springer, Berlin, 1991,
  \url{http://dx.doi.org/10.1007/BFb0090214}, pp.~72--77. \MR{1126739
  (92m:60017)}

\bibitem{Ledoux:1996}
M.~Ledoux, \emph{On {T}alagrand's deviation inequalities for product measures},
  ESAIM Probab. Statist. \textbf{1} (1995/97), 63--87 (electronic),
  \url{http://dx.doi.org/10.1051/ps:1997103}. \MR{1399224 (97j:60005)}

\bibitem{Ledoux:2001}
\bysame, \emph{The {C}oncentration of {M}easure {P}henomenon}, Mathematical
  Surveys and Monographs, vol.~89, American Mathematical Society, Providence,
  RI, 2001. \MR{1849347 (2003k:28019)}

\bibitem{LedouxTalagrand:1991}
M.~Ledoux and M.~Talagrand, \emph{Probability in {B}anach {S}paces:
  Isoperimetry and {P}rocesses}, Ergebnisse der Mathematik und ihrer
  Grenzgebiete (3) [Results in Mathematics and Related Areas (3)], vol.~23,
  Springer-Verlag, Berlin, 1991. \MR{1102015 (93c:60001)}

\bibitem{Levy:1951}
P.~L{\'e}vy, \emph{Probl{\`e}mes {C}oncrets d'{A}nalyse {F}onctionnelle. avec
  un compl{\'e}ment sur les fonctionnelles analytiques par {F}. {P}ellegrino.},
  second ed., Gauthier-Villars, Paris, 1951. \MR{0041346 (12,834a)}

\bibitem{Lugosi:2009}
G.~Lugosi, \emph{Concentration-of-measure inequalities},
  \url{http://www.econ.upf.edu/~lugosi/anu.pdf}, Pompeu Fabra University,
  Barcelona, Spain, 25 June 2009, Lecture notes.

\bibitem{Marton:1996}
K.~Marton, \emph{Bounding {$\overline d$}-distance by informational divergence:
  a method to prove measure concentration}, Ann. Probab. \textbf{24} (1996),
  no.~2, 857--866, \url{http://dx.doi.org/10.1214/aop/1039639365}. \MR{1404531
  (97f:60064)}

\bibitem{McDiarmid:1989}
C.~McDiarmid, \emph{On the method of bounded differences}, Surveys in
  {C}ombinatorics, 1989 ({N}orwich, 1989), London Math. Soc. Lecture Note Ser.,
  vol. 141, Cambridge Univ. Press, Cambridge, 1989, pp.~148--188. \MR{1036755
  (91e:05077)}

\bibitem{McDiarmid:1997}
\bysame, \emph{Centering sequences with bounded differences}, Combin. Probab.
  Comput. \textbf{6} (1997), no.~1, 79--86,
  \url{http://dx.doi.org/10.1017/S0963548396002854}. \MR{1436721 (98b:60020)}

\bibitem{McDiarmid:1998}
\bysame, \emph{Concentration}, Probabilistic {M}ethods for {A}lgorithmic
  {D}iscrete {M}athematics, Algorithms Combin., vol.~16, Springer, Berlin,
  1998, pp.~195--248. \MR{1678578 (2000d:60032)}

\bibitem{MilmanSchechtman:1986}
V.~D. Milman and G.~Schechtman, \emph{Asymptotic {T}heory of
  {F}inite-{D}imensional {N}ormed {S}paces}, Lecture Notes in Mathematics, vol.
  1200, Springer-Verlag, Berlin, 1986, With an appendix by M. Gromov.
  \MR{856576 (87m:46038)}

\bibitem{Rudin:1991}
W.~Rudin, \emph{Functional {A}nalysis}, second ed., International Series in
  Pure and Applied Mathematics, McGraw-Hill Inc., New York, 1991. \MR{1157815
  (92k:46001)}

\bibitem{SullivanTopcuMcKernsOwhadi:2010}
T.~J. Sullivan, U.~Topcu, M.~McKerns, and H.~Owhadi, \emph{Uncertainty
  quantification via codimension-one partitioning}, Int. J. Numer. Meth. Eng.
  \textbf{In Press} (2010), \url{http://dx.doi.org/10.1002/nme.3030}.

\bibitem{Talagrand:1988}
M.~Talagrand, \emph{An isoperimetric theorem on the cube and the
  {K}intchine--{K}ahane inequalities}, Proc. Amer. Math. Soc. \textbf{104}
  (1988), no.~3, 905--909, \url{http://dx.doi.org/10.2307/2046814}. \MR{964871
  (90h:60016)}

\bibitem{Talagrand:1995}
\bysame, \emph{Concentration of measure and isoperimetric inequalities in
  product spaces}, Inst. Hautes \'Etudes Sci. Publ. Math. \textbf{81} (1995),
  73--205, \url{http://dx.doi.org/10.1007/BF02699376}. \MR{1361756 (97h:60016)}

\bibitem{Varadhan:2008}
S.~R.~S. Varadhan, \emph{Large deviations}, Ann. Probab. \textbf{36} (2008),
  no.~2, 397--419, \url{http://dx.doi.org/10.1214/07-AOP348}. \MR{2393987
  (2009d:60070)}

\bibitem{Vu:2002}
V.~H. Vu, \emph{Concentration of non-{L}ipschitz functions and applications},
  Random Structures Algorithms \textbf{20} (2002), no.~3, 262--316,
  Probabilistic methods in combinatorial optimization. \MR{1900610
  (2003c:60053)}

\end{thebibliography}
